\documentclass[12pt,a4paper]{amsart}
\usepackage[centertags]{amsmath}
\usepackage{amsfonts}
\usepackage{amssymb}
\usepackage{amsthm}
\usepackage{newlfont}
\usepackage{amscd}
\usepackage{mathrsfs}

\newcommand{\RR}{\mathbb{R}}
\newcommand{\ZZ}{\mathbb{Z}}
\newcommand{\NN}{\mathbb{N}}

\newcommand{\la}{\lambda}

\newtheorem{thrm}{Theorem}[section]
\newtheorem{lemma}{Lemma}[section]
\newtheorem{cor}{Corollary}[section]
\newtheorem{prop}{Proposition}[section]

\theoremstyle{definition}
\newtheorem{example}{Example}[section]
\newtheorem{rem}{Remark}[section]
\newtheorem{defn}{Definition}[section]

\begin{document}

\title[LP property]{The linear span of projections in AH algebras and for inclusions of $C^*$-algebras}

\author[D. T. Hoa]{Dinh Trung Hoa}
\address{Center of Research and Development, Duy Tan University, K7/25 Quang Trung, Da Nang, Viet Nam}
\email{dinhtrunghoa@duytan.edu.vn}

\author{Toan Minh Ho}
\address{Institute of Mathematics, VAST, 18 Hoang Quoc Viet, Ha Noi, Viet Nam}
\email{hmtoan@math.ac.vn}

\author{Hiroyuki Osaka$^*$}
\date{18, Oct., 2012}
\thanks{$^*$Research of the third author partially supported by the JSPS grant for Scientific Research No.23540256}
\address{ Department of Mathematical Sciences\\
  Ritsumeikan University\\ Kusatsu, Shiga 525-8577,  Japan}

\email[]{osaka@se.ritsumei.ac.jp}

%----------------------------------------------------------------
%\thanks{The author is supported by NAFOSTED, Vietnam.}


%\date{ December 2004.}

%\commby{XXXX}

%---------------------------------------------------------------

\begin{abstract}
A $C^*$-algebra is said to have the LP property if the linear span of projections is dense in a given algebra.
%An inductive limit of matrix algebras over $C^*$-algebras $A = \underrightarrow{\lim}(M_{n_i}(A_i),\phi_i)$ has the LP property if and
%only if each $A_i$ has a spanning set whose image belongs to the closure of the linear span of projections in $A$. 
In the first part of this paper, we show that an AH algebra $A = \underrightarrow{\lim}(A_i,\phi_i)$ has the LP property if and
only if every real-valued continuous function on the spectrum of $A_i$ 
(as an element of $A_i$ via the non-unital embedding) belongs to the closure of the linear span of projections in $A$. 
As a consequence, a diagonal AH-algebra has the LP property if it has small eigenvalue variation.
The second contribution of this paper is that for an inclusion of unital $C^*$-algebras $P \subset A$ with 
a finite Watatani Index, if a faithful conditional expectation $E\colon A \rightarrow P$ has the Rokhlin property 
in the sense of Osaka and Teruya, then $P$ has the LP property under the condition $A$ has the LP property.
As an application, let $A$ be a simple unital $C^*$-algebra with the LP property, $G$ a finite group and $\alpha$ an action 
of $G$ onto $\mathrm{Aut}(A)$. If $\alpha$ has the Rokhlin property in the sense of Izumi, then 
the fixed point algebra $A^G$ and the crossed product algebra $A \rtimes_\alpha G$ have the LP property. 
We also point out that there is a symmetry on CAR algebra, which is constructed by Elliott, 
such that its fixed point algebra does not have the LP property.
\end{abstract}

\maketitle
%%%%%%%%%%%%%%%%%%%%%%%%%%%%%%%%%%%%%%%%%%%%%%%%%%%%
\section{Introduction}
%%%%%%%%%%%%%%%%%%%%%%%%%%%%%%%%%%%%%%%%%%%%%%%%%%%%
A $C^*$-algebra is said to have {\it the LP property} if the
linear span of projections (i.e., the set of all linear
combinations of projections in the algebra) is dense in this
algebra. A picture of the problem which asks to characterize 
the simple $C^*$-algebras to have the LP property was considered in \cite{GP}. 
The LP property of a $C^*$-algebra is weaker than real rank zero since the later means
every self-adjoint element can be arbitrarily closely approximated
by linear combinations of orthogonal projections in this
$C^*$-algebra. In the class of simple AH algebras with slow
dimension growth, real rank zero and small eigenvalue variation
are equivalent (see \cite{bbek}, \cite{be}). Without dimension
hypothesis of AH algebras, we have not known the relations between
these properties.

The concept of {\it diagonal AH algebras} (AH algebra which can be
written as an inductive limit of homogeneous $C^*$-algebras with
diagonal connecting maps) was introduced in \cite{EHT} or
\cite{Ho}. Let us denote by $\mathcal{D}$ the class of diagonal AH
algebras. AF-, AI- and AT- algebras, Goodearl algebras \cite{Goodearl}
and Villadsen algebras of the first type \cite{TW} are diagonal AH algebras. 
Specially, the algebras constructed by Toms in \cite{tom} which have the same K-groups and tracial data but
different Cuntz semigroups are Villadsen algebras of the first
type and so belong to $\mathcal{D}$. This means that the class
$\mathcal{D}$ contains ``ugly'' and interesting $C^*$-algebras and
has not been classified by Elliott's program so far.

Note that the classification program of Elliott, the goal of which
is to classify amenable $C^*$-algebras by their K-theoretical
data, has been successful for many classes of $C^*$-algebras, in
particular for simple AH algebras with slow dimension growth (see,
e.g., \cite{Ell1}, \cite{EG}, \cite{EGL}). Unfortunately, for AH
algebras with higher dimension growth, very few information has
been known.

\vskip 3mm

In the first part of this paper (Section 2), we consider the LP property of inductive limits of
matrix algebras over $C^*$-algebras. The necessary and sufficient conditions for 
such an inductive limit to have the LP property will be presented in Theorem \ref{thrm1}. 
In particular, we will show
that an AH algebra $A = \underrightarrow{\lim}{(A_i,\phi_i)}$
(need not be diagonal nor simple) has the LP property provided
that the image of every continuous function on the spectrum of the
building blocks $A_i$ can be approximated by a linear combination of
projections in $A$ (Corollary \ref{cor1}). In Subsection \ref{bongbong}, 
using the idea of bubble sort, we can rearrange the entries on a diagonal element in $M_n(C(X))$ 
to obtain a new diagonal element with increasing entries such that the eigenvalue variations are the same 
(Lemma \ref{bubblesort}) and the eigenvalue variation of the latter element is easy to evaluate. 
As a consequence, it will be shown that a diagonal AH algebra has the LP property if it has small eigenvalue property 
(Theorem \ref{thrm2}) \textit{without any condition on the dimension growth}.

\vskip 3mm

It is well-known that the LP property of a $C^*$-algebra $A$ is inherited to the matrix tensor product $M_n(A)$ 
and the quotient $\pi(A)$ for any *-homomorphism $\pi$. 
But it is not stable under the hereditary subalgebra of $A$. 
In the second part of this paper (Section 3), we will present the stability of the LP property of an inclusion of a unital $C^*$-algebra 
with certain conditions and some examples illustrated the instability of such the property. 
More precisely, let $1 \in P \subset A$ be an inclusion of unital $C^*$-algebras with a finite Watatani index 
and $E \colon A \rightarrow P$ be a faithful conditional expectation. 
Then the LP property of $P$ can be inherited from that of $A$ provided that $E$ has the Rokhlin property 
in the sense of Osaka and Teruya (Theorem \ref{thm:index LP}). 
As a consequence, given a simple unital $C^*$-algebra $A$ with the LP property 
if an action $\alpha$ of a finite group $G$  onto $\mathrm{Aut}(A)$ has the Rokhlin property 
in the sense of Izumi, then the fixed point algebra $A^G$ and the crossed product $A \rtimes_\alpha G$ 
have the LP property (Theorem \ref{thm:crossed product LP}). 
Furthermore, we also give an example of a simple unital $C^*$-algebra with the LP property, 
but its fixed point algebra does not have the LP property (Example \ref{exa:fixed point}).

Let us recall some notations. Throughout the paper, $M_n$ stands for the algebra of all $n\times
n$ complex matrices, $e=\{ e_{st}\}_{s,t=\overline{1,n}}$ denotes
the standard basis of $M_n$ (for convenience, we also use this
system of matrix unit for any size of matrix algebras). Let us
denote by $M_n(C(X))$ the matrix algebra with entries from the algebra
$C(X)$ of all continuous functions on space $X$. If $X$ has
finitely many connected components $X_i$ and $ X = \sqcup_{i=1}^k X_i$, then
$$M_n(C(X)) = \oplus_{i=1}^k M_n(C(X_i)).$$
Hence, without lost of generality we can always assume the
spectrum of each component of a homogeneous $C^*$-algebra are
connected.

Denoted by $\textrm{diag}(a_{1}, a_{2}, \ldots, a_{n})$ the block
diagonal matrix with entries $a_1, a_2, \ldots, a_n$ in some algebras.

Let $A$ be a $C^*$-algebra. Any element $a$ in $A$ can be
considered as an element in $M_n(A)$ via the embedding $a
\rightarrow \textrm{ diag}(a,0)$. We also denote by $L(A)$ the closure of the set of all linear combinations of finitely many projections in $A$.

The last two authors appreciate Duy Tan University for their warm hospitality 
when they visited September, 2012 and the third author 
would also like to thank Teruya Tamotsu for fruitful discussions about the $C^*$-index theory.

\section{Linear span of projections in AH algebras}
\subsection{Linear span of projections in an inductive limit of matrix algebras over $C^*$-algebras}
Let $$A= \underrightarrow{\lim}{(A_i,\varphi_i)},$$ where $A_i $ = $\oplus_{t=1}^{k_i} M_{n_{it}}(B_{it})$ and $B_{it}$ are
$C^*$-algebras. Let $S_{it}$ be a spanning set of $B_{it}$ (as a
vector space) and $S_i$ be the union of $S_{it}$ for $t=1, \ldots,
k_i$. Since every element of a $C^*$-algebra can be written as a
sum of two self-adjoint elements, we can assume that all elements
of $S_i$ are self-adjoint.

\vskip 3mm

\begin{thrm}\label{thrm1}
Let $A$ be an inductive limit $C^*$-algebra as above. Then the
followings are equivalent:
\begin{enumerate}
\item[(i)] $A$ has the LP property;

\item[(ii)] for any integer $i$, any $x\in S_i$ and any
$\varepsilon>0$, there exists an integer $j\ge i$ such that
$\varphi_{ij}(x)$ can be approximated by an element in $L(A_j)$ to
within $\varepsilon$;

\item[(iii)] for any integer $i$, there exist a spanning set of
$A_i$ such that the images of all elements in that spanning set
under $\varphi_{i\infty}$ belong to $L(A)$.
\end{enumerate}
\end{thrm}

\begin{proof}
The implication (iii) $\Longrightarrow$ (i) is obvious.

To prove the implication (i) $\Longrightarrow$ (ii), it suffices
to mention that every element (projection) in $A$ can be
arbitrarily closely approximated by elements (projections,
respectively) in $A_i$.

Let us prove the implication (ii) $\Longrightarrow$ (iii).
Clearly, without lost of generality we can assume that $A_i  = M_{n_i}(B_i)$ for every $i$. For a fixed integer $i$, we put
$$e\otimes S_{i} = \{e_{st}\otimes x + e_{ts}\otimes
x^*, x\in S_{i}\}.$$ Hence, there exists a unitary $u \in
M_{n_i}$ such that
\begin{equation}\label{equation1}
 u(e_{st}\otimes x + e_{ts}\otimes x)u^* = \left( \begin{array}{ccc}
                  0 & x & 0 \\
                   x  & 0 & 0\\
                   0 & 0 & 0
                  \end{array} \right),
\end{equation}
where the $0$ in the last collum and the last row is of order
$n_i-2$.

It is evident that every element in $A_i$ is a linear combination
of elements in $e\otimes S_{i} \cup D$, where $D$ is the set of
all diagonal elements with coefficients in $S_i$. Thus, $e\otimes
S_{i} \cup D$ is the spanning set of $A_i$. Now, we claim that
this spanning set satisfies the requirement of (iii).

Firstly, let $d$ = diag($x_1, \ldots, x_{n_i}$)\ ($x_t \in S_i$)
be an element in $D$. By (ii), $\varphi_{i \infty}(x_t) \in L(A)$ for every $t$. Hence
$\varphi_{i \infty}(d) \in L(A)$. Lastly, let $a \in e\otimes S_i$. By
Identity (\ref{equation1}), $a$ can be assumed to be
\begin{equation*}
\left( \begin{array}{ccc}
                  0 & x & 0 \\
                   x  & 0 & 0\\
                   0 & 0 & 0
                  \end{array} \right) \quad \text{for some}\quad x \in
                  S_i.
\end{equation*}
Moreover,
\begin{equation*}
a = u^* \left( \begin{array}{ccc}
                  -x & 0 & 0 \\
                   0  & x & 0\\
                   0 & 0 & 0
                  \end{array} \right) u, \quad \textrm{where}
                  \quad
u = \left( \begin{array}{ccc}
                  \frac{1}{\sqrt{2}} & \frac{1}{\sqrt{2}} & 0 \\
                   -\frac{1}{\sqrt{2}} & \frac{1}{\sqrt{2}} & 0\\
                   0 & 0 & 1
                  \end{array} \right).
                  \end{equation*}
In addition, there exists an integer $j\ge i$ such that
$\varphi_{ij}(x)$ can be approximated by an element of $L(A_j)$ to
within $\varepsilon$. Hence $\varphi_{ij}(a)$ can be approximated
by an element of $L(A_j)$ to within $\varepsilon$.
\end{proof}

%\vskip 3mm

\begin{cor}\label{cor1}
Let $A= \underrightarrow{\lim}{(A_i,\varphi_i)}$ be an AH algebra,
where $A_i = \oplus_{t=1}^{k_i} M_{n_{it}}(C(X_{it}))$ and
$X_{it}$ are connected compact Hausdorff spaces. Then the
followings are equivalent:
\begin{enumerate}
\item[(i)] $A$ has the LP property;

\item[(ii)] for any integer $i$, any $f \in
\cup_{t=1}^{k_i}C(X_{it})$ and any $\varepsilon>0$, there exists
an integer $j\ge i$ such that $\varphi_{ij}(f)$ can be
approximated by an element in $L(A_j)$ to within $\varepsilon$.
\end{enumerate}
\end{cor}

From the proof of Theorem \ref{thrm1}, we can obtain the following.
\begin{cor}
Let $A$ be a $C^*$-algebra. If $A$ has the LP property, then $A \otimes M_n$, $A \otimes K$ have the LP property, where $K$ is the algebra of compact operators on a separable Hilbert space.
\end{cor}
%------------------------------------------------------

\subsection{Linear span of projections in a diagonal AH algebra}
For convenience of the reader, let us recall the notions from \cite{EHT}.
Let $X$ and $Y$ be compact Hausdorff spaces. A $\ast$-homomorphism
$\phi$ from $M_{n}(C(X))$ to $M_{nm+k}(C(Y))$ is said to be {\it
diagonal} if there exist continuous maps $\{ \la_i,\ i
=\overline{1,n}\}$ from $Y$ to $X$ such that
\begin{equation*}\phi (f) = \textrm{diag} (f \circ \la_1, f \circ
\la_2, \ldots, f \circ \la_n,0), \quad f \in M_{n}(C(X)),
\end{equation*}
where $0$ is a zero matrix of order $k$ ($k\ge 0$). If the size $k=0$, the map is unital.\\ 
The $\la_i$ are called the {\it eigenvalue maps} (or simply {\it
eigenvalues}) of $\phi$. The family $ \{ \la_1, \la_2, \ldots,
\la_m \} $ is called {\it the eigenvalue pattern} of $\phi$. In addition, let $p$ and $q$ be projections
in $M_n(C(X))$ and $M_{nm+k}(C(Y))$ respectively. An
$\ast$-homomorphism $\psi$ from $pM_{n}(C(X))p$ to
$qM_{nm+k}(C(Y))q$ is called \emph{diagonal} if there exists a
diagonal $\ast$-homomorphism $\phi$ from $M_{n}(C(X))$ to
$M_{nm+k}(C(Y))$ such that $\psi$ is reduced from $\phi$ on
$pM_n(C(X))p$ and $\phi (p) = q$. This definition can also be
extended to a $\ast$-homomorphism
$$ 
\phi : \oplus_{i=1}^n p_i M_{n_i}(C(X_i))p_i \rightarrow \oplus_{j=1}^m q_j M_{m_j}(C(Y_j))q_j   
$$
by requiring that each partial map
$$ 
\phi^{i,j}: p_i M_{n_i}(C(X_i))p_i \rightarrow q_j M_{m_j}(C(Y_j))q_j 
$$
induced by $\phi$ be diagonal.

%-----------------------------------------------------------------------------------------------

\subsection{Eigenvalue variation}

Suppose that $B$ is a simple AH algebra. Then, $B$ has real rank zero if, and only if, its projections separate the traces provided that 
this algebra has slow dimension growth (see \cite{bdr}). 
This equivalence was first studied when the dimensions of the spectra of the building blocks in the inductive limit decomposition of $B$ are not more than two, 
see \cite{bbek}. %Later it was generalized in \cite{bdr}. 

Let $B$ be a C*-algebra. Suppose that
\begin{displaymath}
B = \bigoplus_{i=1}^{k} C(X_{i}) \otimes M_{n_{i}},
\end{displaymath}
where $X_{i}$ is a connected compact Hausdorff space for every $i$. Set $X$ = $\bigsqcup_{i=1}^{k} X_{i}$. 
The following theorem and notations are quoted from \cite{bbek} and \cite{bdr}.

Let $a$ be any self-adjoint element in $B$. For any $x$ in $X_{i}$, any positive integer $m$, $1 \le m \le n_{i}$, let $\la_{m}$ denote the $m^{\textrm{th}}$ lowest eigenvalue of $a(x)$ counted with multiplicity. So $\la_{m}$ is a function on each $X_{i}$, for $i =1, 2, \ldots, k$. The fact is % By Lemma 1.1 of \cite{thom}, we obtain
\begin{displaymath}
|\la_{m} (x) - \la_{m} (y)| \le ||a (x) - a(y)||.
\end{displaymath}
Hence, $\la_{m}$ is continuous, for $m =1, 2, \ldots, k$ for a given summand of $B$.

The {\it variation of the eigenvalues} of $a$, denoted by $EV(a)$, is defined as the maximum of the nonnegative real numbers
\begin{displaymath}
\sup{ \{ |\la_{m}(x) - \la_{m}(y)|; \  x,y \in X_{i} \} },
\end{displaymath}
over all $i$ and all possible values of $m$.

The {\it variation of the normalized trace} of $a$, denoted by $TV (a)$, is defined as % the maximum of the following numbers
\begin{displaymath}
\sup{\{ |\frac{1}{n_{i}} \sum_{m=1}^{n_{i}}{(\la_{m}(x) - \la_{m} (y))} |; \  x, y \in X_{i} \} } = 
\sup{\{ | tr(a(x)) - tr(a(y)) | \  x, y \in X_{i} \} }
\end{displaymath}
over all $i$,
where $tr$ denotes the normalized trace of $M_{n}$ for any positive integer $n$.

\begin{thrm}[Blackadar, B.; Bratteli, O.; Elliott, G.; Kumjian, A.]\label{rr1}
Let $B$ be an inductive limit of homogeneous C*-algebras $B_{i}$ with morphisms $\phi_{ij}$ from $B_{i}$ to $B_{j}$. Suppose that $B_{i}$ has the form 
\begin{displaymath}
B_{i} = \bigoplus_{t = 1}^{k_{i}} M_{n_{it}}(C(X_{it})),
\end{displaymath}
where $k_{i}$ and $n_{it}$ are positive integers, and $X_{it}$ is a connected compact Hausdorff space for every positive integer $i$ and $1 \le t \le k_{i}$. 
Consider the following conditions:
\begin{enumerate}
 \item The projections of $B$ separate the traces on $B$.
 \item For any self-adjoint element $a$ in $B_{i}$ and $\epsilon > 0$, there is a $j \ge i$ such that 
  \begin{displaymath}
   TV(\phi_{ij}(a)) < \epsilon.
  \end{displaymath}
\item For any self-adjoint element $a$ in $B_{i}$ and any positive number $\epsilon $, there is a $j \ge i$ such that 
  \begin{displaymath}
  EV(\phi_{ij}(a)) < \epsilon.
  \end{displaymath}
\item $B$ has real rank zero.
\end{enumerate}

\begin{itemize}
\item[(i)] The following implications hold in general.
\begin{displaymath}
(4) \ \Rightarrow (3)\ \Rightarrow \  (2) \ \Rightarrow \ (1).
\end{displaymath}
\item[(ii)] If $B$ is simple, then the following equivalences hold.
\begin{displaymath}
(3) \ \Leftrightarrow \ (2) \ \Leftrightarrow \ (1).
\end{displaymath}
\item[(iii)] If $B$ is simple and has slow dimension growth, then all the conditions $(1)$, $(2)$, 
$(3)$ and $(4)$ are equivalent.
\end{itemize}
\end{thrm}

\begin{proof}
The statements (i) and (ii) are proved in Theorem 1.3 of \cite{bbek}. The statement (iii) is an immediate consequence of the statement (ii) and Theorem 2 of \cite{bdr}.
\end{proof}

An AH $C^*$-algebra $B$ is said to have {\it small eigenvalue variation} (see \cite{be}) if $B$ satisfies statement (3) of Theorem \ref{rr1}.

%-------------------------------------------------------------------------------------------

\subsection{Rearrange eigenvalue pattern}\label{bongbong}
In order to evaluate the eigenvalue variation \cite{be} of a diagonal element $a$ = diag($a_1, \ldots, a_n$) in $M_n(C(X))$, we need to rearrange the $a_i$ so that the obtained one $b$ = diag($b_1, \ldots, b_n$) with
$b_1\le b_2\le \ldots \le b_n$ has the same eigenvalue variation of $a$. 

The eigenvalue variations of two unitary equivalent self-adjoint
elements are equal since their eigenvalues are the same. However,
the converse need not be true in general. More precisely, there is a self-adjoint element $h$ in $M_2(C(S^4))$  which is not unitarily equivalent to $\textrm{diag}(\la_1,\la_2)$ but the eigenvalue variations of both elements are equal, where $\la_i$ is the $i^{\textrm{th}}$ lowest eigenvalue of $h$ counted with multiplicity \cite[Section 2]{Kadison}. In general, given a
self-adjoint element $h\in M_n(C(X))$, for each $x\in X$, there is
a (point-wise) unitary $u(x)\in M_n$ such that $h(x)$ = $u(x)\textrm{diag}(\la_1(x), \la_2(x), \ldots, \la_n(x))u^*(x)$,
where $\la_i(x)$ is the $i^{\textrm{th}}$ lowest eigenvalue of
$h(x)$ counted with multiplicity. Denote by $EV(h)$ the eigenvalue
variation of $h$, then $EV(h) = EV(\textrm{diag}(\la_1,\la_2,\ldots, \la_n))$ but $u(x)$ need not
be continuous. The fact is that if $u(x)$ is continuous for any self-adjoint $h$ in $M_n(C(X))$, then $\dim (X)$ is less than 3 \cite{Kadison}.  However, when replacing the equality ``='' by some
approximation ``$\approx$'' and in some spacial cases (diagonal
elements) discussed below, we can get such a continuous unitary without any hypothesis on dimension. Let us see the idea via the following example. 

Let $h$ = $\textrm{diag}(x, 1-x)\in M_2(C[0,1])$. Given
any $1/2> \varepsilon>0$. By \cite[Lemma 2.5]{EHT}, there is a
unitary $u\in M_2(C[0,1])$ such that
\begin{itemize}
\item $u(x) = 1 \in M_2,\  \forall x\in [0, \frac{1}{2}-\varepsilon]$
and \item $u(x) = \left( \begin{array}{cc}
              0 & 1 \\
              1 & 0
                      \end{array} \right),
                      \  \forall x\in [\frac{1}{2}+\varepsilon,1].$
\end{itemize}
Denote by $\la_1,\la_2$ the eigenvalue maps of $h$, i.e.,
\[ \la_1(x) =   \left\{ \begin{array}
                      {r@{\quad \textrm{ if } \quad}l}
              x & x \le \frac{1}{2} \\ 1-x & x> \frac{1}{2}
                      \end{array} \right.  \textrm{ and  }
      \la_2(x) =    \left\{ \begin{array}
                      {r@{\quad \textrm{ if } \quad}l}
              1-x & x \le \frac{1}{2} \\ x & x> \frac{1}{2}.
                      \end{array} \right.
                      \]
Then $EV(h) = EV(\textrm{diag}(\la_1,\la_2)) = \frac{1}{2}$.

It is straightforward to check that $||uhu^*  -
\textrm{diag}(\la_1,\la_2)||\le \varepsilon$.

\vskip 3mm

\begin{lemma}\label{bubblesort}
Let $X$ be a connected compact Hausdorff space and $h$ = $\textrm{diag}(f_1, f_2, \ldots, f_n)$ be a self-adjoint element
in $M_n(C(X))$, where $f_1, f_2, \ldots, f_n$ are continuous maps
from  $X$ to $\RR$. For any positive number $\varepsilon$, there is a
unitary $u \in M_n(C(X))$ such that
\[ ||uhu^* - \textrm{diag}(\la_1,\la_2, \ldots, \la_n) ||< \varepsilon, \]
where the $\la_i(x)$ is the $i^{\textrm{th}}$ lowest eigenvalue of
$h(x)$ counted with multiplicity for every $x\in X$.
\end{lemma}
\begin{proof}
If $f_1 \le f_2 \le \ldots \le f_n$, then the unitary $u$ is just the identity of
$M_n$ and $\la_i = f_i$. Therefore, to prove the lemma, we,
roughly speaking, only need to {\it rearrange} the given family
$\{ f_1, f_2, \ldots, f_n\}$ to obtain an increasing ordered
family. For $n=1$, the lemma is obvious. Otherwise, using the idea of
bubble sort, we can reduce to the case $n = 2$. 

Let $Z = {(\la_1-\la_2)}^{-1}(-\frac{\varepsilon}{2}, \frac{\varepsilon}{2})$.
Set
$E = \{ x \in X : f_{1}(x) \le f_{2}(x)  \} \cap (X \backslash Z) $
 and $ F = \{ x \in X : f_{1}(x) \ge f_{2}(x)  \} \cap
(X \backslash Z).$

It is clear that $E, F$ are disjoint closed sets and $X = E \cup F \cup Z$. 
We have $\la_1(x) = \min{\{f_1(x),f_2(x)\}}$ and $\la_2(x) = \max{\{ f_1(x),f_2(x)\}}$
for all $x\in X$. If $E$ ($F$) is empty, then the unitary $u$ can be chosen as 
$\left(
\begin{array}{cc}
              0 & 1 \\
              1 & 0
\end{array} \right)$ ($1 \in M_2$, respectively). Thus, we can assume
both $E$ and $F$ are non-empty. By Urysohn's Lemma, there is a
continuous map $\mu: X \longrightarrow [0,1]$ such that $\mu$ is equal to $0$ on $E$ and $1$ on $F$. Since
the space of unitary matrices  of $M_2$ is path connected, there is a unitary path $p$ linking 
\begin{equation*}
p(0) = 1 \quad \textrm{to} \quad p(1) = \left( \begin{array}{cc}
              0 & 1 \\
              1 & 0
                      \end{array} \right).
\end{equation*}
Consequently, $u = p \circ \mu$ is a unitary in $M_2(C(X))$ and 
$ \displaystyle u(x)h(x)u^*(x) = \textrm{diag} (\la_1(x),
\la_2(x))$ for all $x \in E\cup F$. 

For $x\in X \backslash (E\cup F) = Z$ we have
\begin{equation*}|\la_1(x)-\la_2(x)|<\frac{\varepsilon}{2} \textrm{  and  } 
|f_i(x)-\la_1(x)|<\frac{\varepsilon}{2}, \quad i=1,2.
\end{equation*} 
Hence,
\begin{equation}\label{equ1}
||\textrm{diag}(\la_1(x),\la_2(x))-\textrm{diag}(\la_1(x),\la_1(x))|<\frac{\varepsilon}{2}
\end{equation}
\textrm{  and  } 
\begin{equation}\label{equ2}
||h(x)-\textrm{diag}(\la_1(x),\la_1(x))|<\frac{\varepsilon}{2}.
\end{equation}
On account to (\ref{equ1}) and (\ref{equ2}) we have
\begin{equation*} 
\begin{split}
||u(x) h(x) u^*(x) &-\textrm{diag}(\la_1(x),\la_2(x))||  \\
& \le ||u(x)[h(x) -\textrm{diag}(\la_1(x),\la_1(x))]u^*(x)|| + \\
&\quad + ||\textrm{diag}(\la_1(x),\la_1(x)) - \textrm{diag}(\la_1(x),\la_2(x)) ||\\
&<\varepsilon.
\end{split}
\end{equation*}
Therefore, $$||uhu^*- \textrm{diag}(\la_1,\la_2)||<\varepsilon.$$
\end{proof}

%\subsection{Main Result}

The main result of this section is the following.

\begin{thrm}\label{thrm2}
Given an AH algebra $A= \underrightarrow{\lim} (A_i,\phi_i)$,
where the $\phi_i$ are diagonal $\ast$-homomorphisms from $A_i$ to
$A_{i+1}$, where $A_i = \oplus_{t=1}^{k_i} M_{n_{it}}(C(X_{it}))$ and the $X_{it}$ are connected compact Hausdorff spaces. If $A$ has
small eigenvalue variation, then $A$ has the LP property.
\end{thrm}
%First of all, we would like to give the idea of the proof of
%Theorem \ref{thrm2} as follows. We can assume $k_i=1$ for every
%$i$. By Corollary \ref{cor1}, we only need to show that for any
%$i$, $\varepsilon >0$ and $f \in C(X_i)$ there is an $j $ such that
%$\phi_{ij}(f)$ can be approximated by an element in $L(A_j)$ to
%within $\varepsilon$ for some $j\ge i$. $\phi_{ij}(f)$ is  a diagonal
%element, let's denoted by $a$ = diag($a_1, \ldots, a_n$) in
%$M_n(C(X))$ (for simplicity, let assume $n_j =n, X_j = X$).
%Consider the case that each continuous function $a_i$ has radius
%of its range smaller than $\varepsilon$, that is, \[ |a_i(x)-a_i(y)|
%< \varepsilon, \forall x, y \in X, \ i = 1, 2,\ldots, n.\] Then for
%any $x\in X$, clearly $a$ can be approximated by diag($a_1(x),
%\ldots, a_2(x)$) to within $\varepsilon$. The later element is a
%linear combination of $e_{11}, e_{22}, \ldots, e_{nn}$.  For
%general case, if the variation of $a$ is smaller than $\varepsilon$,
%we need to rearrange the $a_i$ so that we obtain a new diagonal
%element $b$ = diag($b_1, \ldots, b_n$) such that:
%\begin{itemize}
%\item $||a-b||<\varepsilon$, and, \item $|b_i(x)-b_i(y)|<\varepsilon$,
%for every $x,y \in X$ and $i=1, \ldots, n$.
%\end{itemize}
%Since $b$ is approximated by a linear combination of projections to within $\varepsilon$, $a$ is approximated such linear combination to within $2\varepsilon$.% Actually, $b_i(x)$ is the $i^{\textrm{th}}$ smallest eigenvalue of $a(x)$ and $b$ is obtained from $a$ by reaaranging its eigenvalues.
\begin{proof}
By Corollary \ref{cor1}, it suffices to show that $\phi_{i
\infty}(f) \in L(A)$ for every real-valued function $f\in
C(X_{it})$. By the same argument in the proof of Theorem
\ref{thrm1}, we can assume that each $A_t$ has only one component,
that is, $A_t = M_{n_t}(C(X_t))$. Let $\varepsilon > 0$ be arbitrary.
Since $A$ has small eigenvalue variation, there is an integer $j
\ge i$ such that $EV(\phi_{ij}(f))<\varepsilon$. Let $\{\mu_1,\ldots,
\mu_n\}$ be the eigenvalue pattern of $\phi_{ij}$ ($n = n_j/n_i$).
Then,
\begin{equation*}
\begin{split}
\phi_{ij}(f) &= \phi_{ij}(\textrm{diag}(f,0)) \\
&= \textrm{diag}(f \circ \mu_1, 0, f \circ \mu_2, 0, \ldots, f \circ \mu_n,0 ) \\
&= v \textrm{diag}(f_1, f_2, \ldots, f_n,0) v^*,  \end{split}
\end{equation*}
where $ f_i = f \circ \mu_i$ and $v$ is the permutation matrix in
$M_{n_j}$ moving all the zero to the bottom left-hand corner.
Note that 
\begin{equation*}
EV(\phi_{ij}(f)) = EV(\textrm{diag}(f_1, f_2, \ldots, f_n,
f_{n+1})),\end{equation*} 
where $f_{n+1}(x) = 0$ for all $x\in X_j$. By Lemma
\ref{bubblesort}, there exists a unitary $u \in M_{n+1}(C(X_j))$
and eigenvalue maps $\la_1\le \la_2\le \ldots \le \la_{n+1}$ of
$\textrm{diag}(f_1, f_2, \ldots, f_n, f_{n+1})$ such that
\[|| u \textrm{diag}(f_1, f_2, \ldots, f_n, f_{n+1}) u^* - \textrm{diag}(\la_1, \ldots, \la_{n+1})  ||<\varepsilon.\]

Put $$\displaystyle \delta_i = \frac{1}{2}(\max_{x\in X_j}{\la_i(x)}+
\min_{x\in X_j}{\la_i(x)}), \quad i=1,2,\ldots, n+1.$$ Then for any $i$ we have
$$\max\limits_{x\in X_j}{\la_i(x)} - \min\limits_{x\in
X_j}{\la_i(x)} \le EV(\phi_{ij}(f)) <\varepsilon $$
and so
$$|\la_i(x)-\delta_i|<\varepsilon, \quad \forall x\in X_j.$$ 
Thus,
\[|| u \textrm{diag}(f_1, f_2, \ldots, f_n, f_{n+1}) u^* - \sum_{i=1}^{n+1} \delta_i e_{ii}  ||<2 \varepsilon, \]
where $\{e_{ij}\}$ is the standard basis of $M_{n+1}$. This
implies that $$||\phi_{ij}(f)- b||<2\varepsilon,$$ where $b = v^*
\textrm{diag}(u^*(\sum_{i=1}^{n+1} \delta_i e_{ii}) u,0) v$ is a linear combination of projections in $A_j$. 

Therefore,
$$\phi_{i \infty}(f) \in L(A).$$
\end{proof}

\subsection{Another form of Theorem \ref{thrm2}.}

\begin{lemma} \label{cutdown}
Let $B$ be a $C^*$-algebra,  $p$ and $q$ be projections
in $B$. If $p$ and $q$ are Murray-von Neumann equivalent, then $p
B p$ is isomorphic to $q B q$.

In particular, if $B = M_n(C(X))$ (where $X$ is a connected compact Hausdorff
space) and $q$ is a constant projection of rank
$m$ in $B$, then $qBq$ is $\ast$-isomorphic to $M_m(C(X))$.
\end{lemma}
\begin{proof}
By assumption, there exists a partial isometry $v$ such that $p $ = $v^{*} v $ and
$q$ = $v v^{*}$. Let us condiser the following maps 
$$\phi (x) = v x v^{*}\ (x \in p B p) \textrm{  and  } \psi (y) = v^{*} y v\ (y \in qBq).$$
It is straightforward to check that the compositions of $\phi$ and $\psi$ are the identity maps.

In the case $B = M_n(C(X))$ and $q$ is a constant projection of rank $m$ in
$B$, we have $qBq=M_m(C(X))$. Therefore, $pBp$ is $\ast$-isomophic to $M_m(C(X)).$
\end{proof}

\vskip 3mm

\begin{thrm}[Another form of Theorem \ref{thrm2}]
Let $A= \underrightarrow{\lim}(A_i, \phi_i)$ be a diagonal AH
algebra, where the $p_{it}$ are projections in $M_{n_{it}}(C(X_{it}))$, 
$A_i = \oplus_{t=1}^{k_{i}} p_{it} M_{n_{it}}(C(X_{it}))p_{it}$ and the $\phi_i$ are unital diagonal. 
Suppose that each projection $p_{1t}$ is Murray-von
Neumann equivalent to some constant projection in $A_1$, then $A$
has the LP property provided that $A$ has small eigenvalue
variation.
\end{thrm}
\begin{proof}
We can assume that $A_i = p_iM_{n_i}(C(X_i))p_i, \ \forall i$. It is easy to see that $p_1$ is Murray-von
Neumann equivalent to $ q_1=e_{11}+e_{22}+\ldots + e_{m_1}$, where $m_1$ is the rank of $p_1$. For $i>1$, define $q_i = \phi_{i-1}(q_{i-1})$. Then
$q_i = \phi_{1i}(q_1)$ is constant, since $q_1$ is constant and
$\phi_{1i}$ is diagonal. Let us denote by $m_i $ the rank of
$q_i$, then $m_{i+1}|m_i$. By Lemma \ref{cutdown}, there are
$\ast$-isomorphisms $\Theta_i$ from $p_i M_{n_i}(C(X_i))p_i$ to $q_i M_{n_i}(C(X_i))q_i = M_{m_i}(C(X))$ such that
\begin{equation*}
\Theta_i(a) = v_i a v_{i}^*, 
\end{equation*}
where $p_1 = v_1^*v_1 \sim v_1v_1^* = q_1 $ and $v_i =\phi_{1i}(v_1)$. Since $\phi_i$ is diagonal, 
there exists its extension $\tilde{\phi_i}$ which is a diagonal $\ast$-homomorphism from $M_{n_i}(C(X_i))$ to $M_{n_{i+1}}(C(X_{i+1}))$. Let $\psi_i$ be the restriction of $\tilde{\phi_i}$ on
$q_iM_{n_i}(C(X_i))q_i$. Then $\psi_i (q_i) = q_{i+1}$. Therefore, the map $\psi_i$ can be viewed as the map
from $q_iM_{n_i}(C(X_i))q_i$ to $q_{i+1}M_{n_{i+1}}(C(X_{i+1}))q_{i+1}$ and so $\underrightarrow{\lim}(M_{m_i}(C(X_i)),\psi_i)$ is a diagonal AH-algebra. 

On another hand, it is straightforward to check that  $\Theta_{i+1}\circ \phi_i = \psi_i \circ \Theta_i$ and hence $A=\underrightarrow{\lim}(M_{m_i}(C(X_i)),\psi_i)$. By
Theorem \ref{thrm2}, $A$ has the LP property.
\end{proof}

\subsection{Examples}

In some special cases, the small eigenvalue variation and the LP property are equivalent. 

\vskip 3mm

\begin{example}\label{GoodearlLP}
Let $A= \underrightarrow{\lim}(M_{\nu (n)}(C(X)),\phi_n)$ be a Goodearl algebra \cite{Goodearl} and $\omega_{t,1}$ be
the weighted identity ratio for $\phi_{t,1}$. Suppose that $X$ is not totally disconnected and has finitely many connected components, the followings are equevalent:
\begin{itemize}
\item[(i)] $A$ has real rank zero. 
\item[(ii)] $\displaystyle
 \lim_{t\rightarrow \infty}{\omega_{t,1}} = 0$. 
\item[(iii)] $A$ has small eigenvalue variation. 
\item[(iv)] $A$ has the LP property.
\end{itemize}
\end{example}

\begin{proof}
Indeed, (i) and (ii) are equivalent by \cite[Theorem 9]{Goodearl} . The implication (i) $\Longrightarrow$
(iii) follows from \cite[Theorem 1.3]{bbek}. By \cite[Theorem 2.6]{BP}, (i) implies (iv). Using Theorem
\ref{thrm2} we get the implication (iii) $\Longrightarrow$ (iv). Finally, (iv) implies (ii) by \cite[Theorem 6]{Goodearl}.
\end{proof}

\vskip 3mm

In general, the LP property can not imply the small eigenvalue variation nor real rank zero.  For example, let $A$ be a simple AH algebra with slow
dimension growth and $H$ be a simple hereditary $C^*$-subalgebra of $A$. By \cite[Theorem 3.5]{Ho}, $H$ has non-trivial
projections. Hence, $H \otimes K$ has the LP property by
\cite[Corollary  5]{GP} . However, $A$ has real rank zero if and only if it has small eigenvalue variation
\cite{be}. It means that we can choose $H$ such that the real rank of $H$ is not zero. 

Looking for examples in the class of diagonal AH algebras, we need the following lemma.

\begin{lemma}\label{tensordiagonal}
Let $A$ be a diagonal AH algebra and  $K$ be the $C^*$-algebra of compact operators on an infinite dimensional Hilbert space. Then the tensor product $A\otimes K$ is again diagonal.
\end{lemma}
\begin{proof}
Let $A$ = $ \underrightarrow{\lim}(A_n, \phi_n)$ and $K$ = $\underrightarrow{\lim}(M_n, i_n)$, where $A_n$ is a homogeneous algebra, $\phi_n$ is an injective diagonal homomorphism from $A_n$ to $A_{n+1}$ and $i_n$ is the embedding from $M_n$ to $M_{n+1}$ which  associates each $a\in M_n$ to diag($a,0$) $ \in M_{n+1}$ for each positive integer $n$. 
Let us consider the inductive limit $ \underrightarrow{\lim}(A_n\otimes M_n, \phi_n\otimes i_n)$. 
For each integer $n\ge 1$, denote by $i_{n,\infty}$, $\phi_{n,\infty}$ the homomorphisms from $M_n$, $A_n$ to $K$, $A$ in the inductive limit of $K$, $A$ respectively. Then 
$$ (\phi_{n+1,\infty}\otimes i_{n+1,\infty})\circ (\phi_n \otimes i_n) = \phi_{n,\infty}\otimes i_{n,\infty}.$$
Hence, by the universal property of inductive limit, there exists a unique homomorphism $\Phi$ from 
$ \underrightarrow{\lim}(A_n\otimes M_n, \phi_n\otimes i_n)$ to $A\otimes K$ such that 
$$ \Phi \circ (\phi_n \otimes i_n) = (\phi_{n,\infty} \otimes i_{n,\infty}). $$
It is straightforward to check that the image of $\Phi$ is dense in $A\otimes K$ and 
since all the maps $\phi_n,i_n$ are injective, we have $ \underrightarrow{\lim}(A_n\otimes M_n, \phi_n\otimes i_n)$ is $A\otimes K$. 
Furthermore, for each $n$, we identify an element $a\otimes b$ in $A_n \otimes M_n$ with the matrix $(a_{ij}b)$ in $M_n(A_n)$, 
where $a=(a_{ij}) \in M_n$ and $b\in A_n$. By interchanging rows and columns (independent of $a\otimes b$) of $(\phi_n\otimes i_n)(a\otimes b)$, 
we obtain diag($a\otimes b \circ \lambda_1, \ldots, a\otimes b \circ \lambda_m,0$), where $\lambda_1, \ldots, \lambda_m$ are the eigenvalue maps of $\phi_n$. 
This means that there is a permutation matrix $u_n\in M_{n+1}(A_{n+1})$ such that $u_n \phi_n\otimes i_n u_n^*$ is diagonal. 
The fact is that the inductive limit is unchanged under unitary equivalence, that is,
$$\underrightarrow{\lim}(A_n\otimes M_n, \phi_n\otimes i_n) = \underrightarrow{\lim}(A_n\otimes M_n, u_n \phi_n\otimes i_n u_n^*). $$
Hence,  $\underrightarrow{\lim}(A_n\otimes M_n, \phi_n\otimes i_n)$ is diagonal.
\end{proof}

\begin{example}
Let $B$ be a simple unital diagonal AH algebra with real rank one but does not have the LP property (for example, take a Goodearl algebra, see Example \ref{GoodearlLP}). 
Then $B\otimes K$ is also a diagonal AH algebra of real rank one with the LP property.
\end{example}
\begin{proof}
By Lemma \ref{tensordiagonal}, $B\otimes K$ is also a diagonal AH algebra. The real rank of $B\otimes K$ is one since that of $B$ is non-zero.
Since $B$ is unital, $B\otimes K$ has a non-trivial projection. By \cite[Corollary  5]{GP}, $B\otimes K$ has the LP property.
\end{proof}

%%%%%%%%%%%%%%%%%%%%%Prof. Osaka

\section{The LP property for an inclusion of unital $C^*$-algebras}
\subsection{Examples}

In this subsection we will show that the LP property does not stable under the fixed point operation via giving examples. 
Firstly, we could observe the following example which shows that the LP property is not stable under the hereditary subalgebra.

\vskip 3mm

\begin{lemma}\label{lem:LP}
Let $A$ be a projectionless simple unital $C^*$-algebra with a unique tracial state. 
Then for any $n \in \NN$ with $n > 1$, $M_n(A)$ has the LP property.
\end{lemma}

\begin{proof}
Note that $M_n(A)$ has also a unique tracial state.

Since $A$ is unital, $M_n(A)$ has a non-trivial projection. 
Then by \cite[Corollary 5]{GP} $M_n(A)$ has the LP property.
\end{proof}

\vskip 3mm

\begin{rem}
Let $A$ be the Jiang-Su algebra. Then we know that $RR(A) = 1$ (\cite{BP}). 
Since $M_n(A)$ is an AH algebra without real rank zero, $RR(M_n(A)) = 1$.
But from Lemma~\ref{lem:LP} $M_n(A)$ has the LP property.
\end{rem}

\vskip 3mm

Using this observation we can construct a $C^*$-algebra with the LP property 
such that the fixed point algebra does not have the LP property.

\vskip 3mm

\begin{example}\label{exa:fixed point}
An simple unital AI algebra $A$ in \cite[Example~9]{Ell2}, which comes from Thomsen's construction,  has two extremal tracial states, so
by \cite[Theorem~4.4]{Thom} $A$ does not have the LP property. There is a symmetry $\alpha$ on $A$ 
constructed by Elliott such that $A \rtimes_\alpha \ZZ/2\ZZ$ is a UHF algebra. Since the fixed point algebra 
$(A \rtimes_\alpha \ZZ/2\ZZ)^\beta = A$, where $\beta$ is the dual action of $\alpha$. This shows that there is a simple unital 
$C^*$-algebra $B$ with the LP property such that the fixed point algebra $B^\beta$ does not have the LP property.
\end{example}

%------------------------------------------------------------------------

\subsection{$C^*$-index Theory}

%In this section we show that the LP property is preserved under formation of 
%crossed products and fixed point algebras by actions of finite groups which have the Rokhlin property.

According to Example~\ref{exa:fixed point} there is a faithful conditional expectation $E\colon B \rightarrow B^\beta$.
We extend this observation to an inclusion of unital $C^*$-algebras with a finite Watatani index as follows.

In this section we recall the 
$C^*$-basic construction defined by Watatani.

\vskip 3mm

\begin{defn}
Let $A \supset P$ be an  inclusion of unital $C^*$-algebras with a conditional expectation $E$ from $A$ onto $P$.
\begin{enumerate}
 \item A {\it quasi-basis} for $E$ is a finite set $\{(u_i, v_i)\}_{i=1}^n \subset A \times A$ such that 
 for every $a \in A$, 
 $$
 a = \sum_{i=1}^nu_iE\left(v_i a\right)= \sum_{i=1}^n E\left(a u_i\right)v_i.
 $$
 \item When $\{(u_i, v_i)\}_{i=1}^n$ is a quasi-basis for $E$, we define $\mathrm{Index} E$ by 
 $$
 \mathrm{Index} E = \sum_{i=1}^n u_iv_i.
 $$
 When there is no quasi-basis, we write $\mathrm{Index} E = \infty$. $\mathrm{Index} E$ is called the 
 Watatani index of $E$.  
\end{enumerate}
\end{defn}

\begin{rem}\label{rmk:quasi} We give several remarks about the above definitions.
\begin{enumerate}
 \item $\mathrm{Index} E$ does not depend on the choice of the quasi-basis in the above formula, 
 and it is a central element of $A$ 
 $($\cite[Proposition 1.2.8]{Watatani:index}$)$.
 \item Once we know that there exists a quasi-basis, we can choose one of the form 
 $\{(w_i, w_i^*)\}_{i=1}^m$, which shows that $\mathrm{Index} E$ is a positive element 
 $($\cite[Lemma 2.1.6]{Watatani:index}$)$.
 \item By the above statements, if $A$ is a simple $C^*$-algebra, then $\mathrm{Index} E$ is a 
 positive scalar.
 %\item Let $\{(u_i, v_i)\}_{i=1}^n$ be a quasi-basis for $E$. 
 %If $A$ acts on a Hilbert space $\mathcal{H}$ faithfully, then we can define the map $E^{-1}$
 %from $P' \cap B(\mathcal {H})$ to $A' \cap B(\mathcal{H})$ by $E^{-1}(x) = \sum_{i=1}^n u_i x v_i$ for $x$ in $P' \cap B(\mathcal{H})$. 
 %In fact, 
 %for any $x \in P'\cap B(\mathcal{H})$ and $a \in A$ 
 %\begin{eqnarray*}
 %E^{-1}(x)a&=& \sum_{i =1}^n u_i x v_i a \\
 %&=&\sum_{i, j =1}^n u_i x E(v_i a u_j)v_j \\
 %&=&\sum_{i, j =1}^n u_i E(v_i a u_j)x v_j \\
 %&=&\sum_{j =1}^n  a u_jx v_j = aE^{-1}(x). \\
%\end{eqnarray*}
\item If $\mathrm{Index} E < \infty$, then $E$ is faithful, i.e., $E(x^*x) = 0$ implies $x=0$ for $x \in A$.
\end{enumerate}
\end{rem}

%\subsubsection{$C^*$-basic construction}

\vskip 3mm

Next we recall the 
$C^*$-basic construction defined by Watatani.

\vskip 3mm

Let $E\colon A\to P$ be a faithful conditional expectation.
Then $A_{P}(=A)$ is
 a pre-Hilbert module over $P$ with a $P$-valued inner
product $$\langle x,y\rangle_P =E(x^{*}y), \ \ x, y \in A_{P}.$$
We denote by ${\mathcal E}_E$ and $\eta_E$ the Hilbert $P$-module completion of $A$ 
by the norm $\Vert x \Vert_P = \Vert \langle x, x \rangle_P\Vert^{\frac{1}{2}}$ for $x$ in $A$
and the natural inclusion map 
from $A$ into ${\mathcal E}_E$.  
Then ${\mathcal E}_E$ is a Hilbert $C^{*}$-module over $P$.
Since $E$ is faithful, the inclusion map $\eta_E$ from  $A$ to ${\mathcal E}_E$ is injective.
Let $L_{P}({\mathcal E}_E)$ be the set of all (right) $P$-module homomorphisms
$T\colon {\mathcal E}_E \to {\mathcal E}_E$ with an adjoint right $P$-module homomorphism
$T^{*}\colon {\mathcal E}_E \to {\mathcal E}_E$ such that $$\langle T\xi,\zeta
\rangle =
\langle \xi,T^{*}\zeta \rangle \ \ \ \xi, \zeta \in {\mathcal E}_E.$$
Then $L_{P}({\mathcal E}_E)$ is a $C^{*}$-algebra with the operator norm
$\|T\|=\sup\{\|T\xi \|:\|\xi \|=1\}.$ There is an injective
$*$-homomorphism $\lambda \colon A\to L_{P}({\mathcal E}_E)$ defined by
$$
\lambda(a)\eta_E(x)=\eta_E(ax)
$$
for $x\in A_{P}$ and  $a\in A$, so that $A$ can
be viewed as a
$C^{*}$-subalgebra of $L_{P}({\mathcal E}_E)$.
Note that the map $e_{P}\colon A_{P}\to A_{P}$
defined by 
$$
e_{P}\eta_E(x)=\eta_E(E(x)),\ \ x\in
A_{P}
$$
 is
bounded and thus it can be extended to a bounded linear operator, denoted
by $e_{P}$ again, on ${\mathcal E}_E$.
Then $e_{P}\in L_{P}({{\mathcal E}_E})$ and $e_{P}=e_{P}^{2}=e_{P}^{*}$, that
is, $e_{P}$ is a projection in $L_{P}({\mathcal E}_E)$.
A projection $e_P$ is called the {\em Jones projection} of $E$.

The {\sl (reduced) $C^{*}$-basic construction} is a $C^{*}$-subalgebra of
$L_{P}({\mathcal E}_E)$, defined as
$$
C^{*}_r\langle A, e_{P}\rangle = \overline{ span \{\lambda (x)e_{P} \lambda (y) \in
L_{P}({{\mathcal E}_E}): x, \ y \in A \ \} }^{\|\cdot \|} 
$$

\begin{rem}\label{rmk:b-const}
Watatani proved the following in \cite{Watatani:index}:
\begin{enumerate}
 \item $\mathrm{Index} E$ is finite if and only if $C^{*}_r\langle A, e_{P}\rangle$ has the identity 
 (equivalently $C^{*}_r\langle A, e_{P}\rangle = L_{P}({\mathcal E}_E)$) and there exists a constant 
 $c>0$ such that $E(x^*x) \geq cx^*x$ for $x \in A$, i.e., $\Vert x \Vert_P^2 \geq c\Vert x \Vert^2 $ 
 for $x$ in $A$ by \cite[Proposition 2.1.5]{Watatani:index}.
 Since $\Vert x \Vert \geq \Vert x \Vert_P$ for  $x$ in $A$, if $\mathrm{Index} E$ is finite, then ${\mathcal E}_E = A$.
 \item If $\mathrm{Index} E$ is finite, then each element $z$ in $C^{*}_r\langle A, e_{P}\rangle$ has a form 
 $$
 z = \sum_{i=1}^n \lambda(x_i) e_P \lambda(y_i)
 $$
 for some $x_i$ and $y_i $ in $A$.
 \item Let $C^{*}_{\max}\langle A, e_{P}\rangle$ be the unreduced $C^*$-basic construction defined in 
 Definition 2.2.5 of \cite{Watatani:index}, which has the certain universality (cf.(5)).
 If $\mathrm{Index} E$ is finite, then there exists an isomorphism from 
 $C^{*}_r\langle A, e_{P}\rangle$ onto $C^{*}_{\max}\langle A, e_{P}\rangle$ (\cite[Proposition 2.2.9]{Watatani:index}).
 Therefore, we  can identify $C^{*}_r\langle A, e_{P}\rangle$ with $C^{*}_{\max}\langle A, e_{P}\rangle$. 
 So we call $C^{*}_r\langle A, e_{P}\rangle$ the $C^*$-{\it basic construction} and denote it by $C^{*}\langle A, e_{P}\rangle$. 
 Moreover, 
 we identify $\lambda(A)$ with $A$ in $C^*\langle A, e_p\rangle (= C^{*}_r\langle A, e_{P}\rangle)$, 
and we define it as 
$$
C^*\langle A, e_p\rangle = \{ \sum_{i=1}^n x_i e_P y_i : x_i, y_i \in A, n \in \NN\}.
$$
\item The $C^*$-basic construction 
$C^{*}\langle A, e_p\rangle$ is isomorphic to  $qM_n(P)q$ 
for some $n \in \NN$ and projection $q \in M_n(P)$ 
$($\cite[Lemma~3.3.4]{Watatani:index}$)$.
If $\mathrm{Index} E$ is finite, then $\mathrm{Index} E$ is a central  invertible  element of $A$ and 
there is the dual conditional expectation $\hat{E}$ from $C^{*}\langle A, e_{P}\rangle$ onto $A$ such that 
$$
   \hat{E}(x e_P y) = (\mathrm{Index} E)^{-1}xy \quad \text{for} \ x, y \in A
 $$    
 by \cite[Proposition 2.3.2]{Watatani:index}. Moreover, $\hat{E}$ has a finite index and faithfulness.
 If $A$ is simple unital $C^*$-algebra, $\mathrm{Index} E \in A$ by Remark~\ref{rmk:quasi}(4). 
 Hence $\mathrm{Index} E = \mathrm{Index} \hat{E}$ by \cite[Proposition~2.3.4]{Watatani:index}.
\item Suppose that $\mathrm{Index} E$ is finite and $A$ acts on a Hilbert space $\mathcal{H}$ faithfully and $e$ is a projection on 
 $\mathcal{H}$ such that $eae =E(a)e$ for $a \in A$. If a map $P \ni x \mapsto xe \in B$ $($$\mathcal{H}$$)$ 
 is injective, then there exists an isomorphism $\pi$ from the norm closure  of a linear span of $AeA$ to  
 $C^{*}\langle A, e_{P}\rangle$ such that $\pi(e) = e_P$ and $\pi(a) = a$ for $a \in A$ \cite[Proposition 2.2.11]{Watatani:index}.
\end{enumerate}
\end{rem}

\vskip 3mm

%---------------------------------------------------------------------------------

\subsection{Rokhlin property for an inclusion of unital C*-algebras}

For a $C^*$-algebra $A$, we set 

\begin{eqnarray*}
c_0(A) &=& \{(a_n) \in l^\infty(\NN, A): \lim\limits_{n \to \infty} \Vert a_n \Vert= 0\}  \\
 A^\infty &=&l^\infty(\NN, A)/c_0(A).  
\end{eqnarray*}

We identify $A$ with the $C^*$-subalgebra of $A^\infty$ consisting of the equivalence classes of constant sequences and 
set
$$
A_\infty = A^\infty \cap A'.
$$
For an automorphism $\alpha \in {\mathrm{ Aut}}(A)$, we denote by $\alpha^\infty$ and $\alpha_\infty$ the automorphisms of 
$A^\infty$ and $A_\infty$ induced by $\alpha$, respectively.

Izumi defined the Rokhlin property for a finite group action in \cite[Definition 3.1]{Iz} as follows:

\begin{defn}\label{def:group action}
Let $\alpha$ be an action of a finite group $G$ on a unital $C^*$-algebra $A$. 
$\alpha$ is said to have the {\it Rokhlin property} if there exists a partition of unity 
$\{e_g\}_{g \in G} \subset A_\infty$ consisting of projections satisfying 
$$
(\alpha_g)_\infty(e_h) = e_{gh} \quad \text{for} \  g, h \in G.
$$
We call $\{e_g\}_{g\in G}$ Rokhlin projections. 
\end{defn}  

\vskip 3mm

Let $A \supset P$ be an inclusion of unital $C^*$-algebras.
For a conditional expectation $E$ from  $A$ onto $P$, we denote by $E^\infty$ the natural
conditional expectation from $A^\infty $ onto $ P^\infty$ induced by $E$.
If $E$ has a finite index with a quasi-basis $\{(u_i, v_i)\}_{i=1}^n$, then 
$E^\infty$ also has a finite index with a quasi-basis $\{(u_i, v_i)\}_{i=1}^n$
and $\mathrm{Index} (E^\infty) = \mathrm{Index} E$.

\vskip 3mm

Motivated by Definition~\ref{def:group action},
Kodaka, Osaka, and Teruya introduced the Rokhlin property for an inclusion of 
unital $C^*$-algebras with a finite index \cite{KOT}. 

\vskip 3mm

\begin{defn}\label{Rokhlin}
A conditional expectation $E$ of a unital  $C^*$-algebra $A$ with a finite index is said to have the {\it Rokhlin property} 
if there exists a  projection $e \in A_\infty$ satisfying 
$$
E^\infty(e) = ({\mathrm{Index}}E)^{-1} \cdot 1
$$
and a map $A \ni x \mapsto xe$ is injective. We call $e$ a Rokhlin projection.
\end{defn}

\vskip 3mm

The following result states that 
the Rokhlin property of an action in the sense of Izumi implies that the canonical 
conditional expectation from a given simple $C^*$-algebra 
to its fixed point algebra has the Rokhlin property in the sense of Definition~\ref{Rokhlin}.

\vskip 1mm

\begin{prop}\label{prp:group}$($\cite{KOT}$)$
Let $\alpha$ be an action of a finite group $G$ on a unital $C^*$-algebra $A$ and 
$E$ be the canonical conditional expectation from $A$ onto the fixed point algebra $P = A^{\alpha}$ defined by 
$$
E(x) = \frac{1}{\#G}\sum_{g\in G}\alpha_g(x) \quad \text{for} \ x \in A, 
$$
where $\#G$ is the order of $G$.
Then $\alpha$ has the Rokhlin property if and only if there is a projection $e \in A_\infty$ such that 
$E^\infty (e) = \frac{1}{\#G}\cdot 1$, where $E^\infty$ is the conditional expectation from $A^\infty$ onto 
$P^\infty$ induced by $E$.
\end{prop}

\vskip 3mm

The following is the key one in the next section.

\begin{prop}\label{prp:embedding}$($\cite{KOT}\cite[Lemma~2.5]{OT}$)$
Let $P \subset A$ be an inclusion of unital $C^*$-algebras and 
$E$ be a conditional expectation from $A$ onto $P$ with a finite index.
If $E$ has the Rokhlin property with a Rokhlin projection $e \in A_\infty$, then 
there is a unital linear map 
$\beta \colon A^\infty 
\rightarrow P^\infty$ such that 
for any $x \in A^\infty$ there exists 
the unique element $y$ of $P^\infty$ such that $xe = ye = \beta(x)e$ and 
$\beta(A' \cap A^\infty) \subset P' \cap P^\infty$. 
In particular, $\beta_{|_A}$ is a unital injective *-homomorphism and 
$\beta(x) = x $ for all  $x \in P$.
\end{prop}

\vskip 3mm

The following is contained in \cite[Proposition~3.4]{KOT}. But 
we give it for self-contained.

\vskip 3mm

\begin{prop}\label{prp:Rokhlin}
Let $P \subset A$ be an inclusion of unital $C^*$algebras and 
$E$ be a conditional expectation from $A$ onto $P$ with a finite index.
Suppose that $A$ is simple.
Consider the basic construction
$$
P \subset A \subset C^*\langle A, e_P\rangle (:= B) \subset C^*\langle B, e_A\rangle (:= B_1).
$$

If $E \colon A \rightarrow P$ has the Rokhlin property with a Rokhlin projection $e \in A_\infty$, then 
the double dual conditional expectation $\hat{\hat{E}} (:= E_B) \colon C^*\langle B, e_A\rangle \rightarrow B$ has the 
Rokhlin property.
\end{prop}

\vskip 3mm

\begin{proof}
Note that from Remark~\ref{rmk:b-const}(4) and \cite[Corollary~3.8]{KOT} $C^*$-algebras 
$C^*\langle A, e_P\rangle$ and $C^*\langle B, e_A\rangle$  are simple.

Since $e_Pee_P = E^\infty(e)e_P = ({\mathrm{Index} E})^{-1}e_p$, 
$(\mathrm{Index} E)ee_pe \leq e$ and 
$$
\hat{E}^\infty(e - (\mathrm{Index} E)ee_pe) = e - (\mathrm{Index} E)e\hat{E}^\infty(e_P)e = e - e = 0,
$$ we have 
$e = (\mathrm{Index} E)ee_Pe$.
Then, for any $x, y \in A$
\begin{align*}
e(xe_Py)e &= exee_Peye\\
&= ({\mathrm{Index} E})^{-1} exye\\
&= \hat{E}(xe_Py)e
\end{align*}
Hence, 
from Remark~\ref{rmk:b-const}(3) we have $eze = \hat{E}(z)e$ for any $z \in C^*\langle A, e_P\rangle$.

Let $\{(w_i, w_i^*)\} \subset B \times B$ be a quasi-basis for $\hat{E} (= E_A)$ and 
$e_A$ be the Jones projection of $\hat{E}$. 
Set $g = \sum_iw_iee_Aw_i^* \in B_1^\infty$. 
Then $g$ is a projection and $g \in B_1'$. Indeed, since
\begin{align*}
g^2 &= \sum_{i,j}w_iee_Aw_i^*w_jee_Aw_j^*\\
&= \sum_{i,j}w_iee_A\hat{E}(w_i^*w_j)w_j^*\\
&= \sum_iw_iee_A(\sum_j\hat{E}(w_i^*w_j)w_j^*)\\
&= \sum_iw_iee_Aw_i^* =g,
\end{align*}
$g$ is a projection. 
\begin{align*}
ge_A &= \sum_iw_iee_Aw_i^*e_A\\
&= \sum_iw_i\hat{E}(w_i^*)ee_A\\
&= ee_A
\end{align*} and 
\begin{align*}
e_Ag &= e_A\sum_iw_iee_Aw_i^*\\
&= \sum_i\hat{E}(w_i)e_Aew_i^*\\
&= e_Ae\sum_i\hat{E}(w_i)w_i^*\\
&= ee_A = ge_A.
\end{align*}

Moreover, for any $z \in C^*\langle A, e_P\rangle$ we have
\begin{align*}
gz &= \sum_iw_iee_Aw_i^*z\\
&= \sum_iw_iee_A(\sum_j\hat{E}(w_i^*zw_j)w_j^*)\\
&= \sum_iw_i\sum_j\hat{E}(w_i^*xzw_j)ee_Aw_j^*\\
&= \sum_j(\sum_iw_i\hat{E}(w_i^*zw_j)ee_Aw_j^*\\
&= \sum_jzw_jee_Aw_j^*\\
&= z\sum_jw_jee_Aw_j^*\\
&= zg
\end{align*}
Since $B_1 = C^*\langle C^*\langle A, e_P\rangle, e_A\rangle$,
$g \in B_1' \cap B_1^\infty$.

To prove that the double dual conditional expectation $\hat{\hat{E}}$ has the 
Rokhlin property, we will show that $g$ is the Rokhlin projection of $\hat{\hat{E}}$.
Since $eze = \hat{E}(z)e$ for any $z \in C^*\langle A, w_P\rangle$,
by Remark~\ref{rmk:b-const}(5), there exists an isomorphism 
$\pi \colon C^*\langle C^*\langle A, e_P\rangle, e_A\rangle\rightarrow C^*\langle C^*\langle A, e_P\rangle, e\rangle$
such that $\pi(e_A) = e$ and $\pi(z) = z$ for $z \in C^*\langle A, e_P\rangle$. Then 

\begin{align*}
\hat{\hat{E}}^{\infty}(g) &= \sum_iw_ie\hat{\hat{E}}^\infty(e_A)w_i^*\\
&= \sum_iw_i({\mathrm{Index} E})^{-1}ew_i^*\\
&= ({\mathrm{Index} E})^{-1}\sum_iw_i\pi(e_A)w_i^*\\
&= ({\mathrm{Index} E})^{-1}\pi(\sum_iw_ie_Aw_i^*)\\
&= ({\mathrm{Index} E})^{-1}1\\
&= ({\mathrm{Index} \hat{\hat{E}}})^{-1}1,
\end{align*}
hence $\hat{\hat{E}}$ has the Rokhlin property.
\end{proof}

%%%%%%%%%%%%%%%%%%%%%%%%%%%%%%%%%%%%%%%%%%%%%%%%%%%%%%%%%%%%%%%%%%%
\subsection{Main results}

\begin{thrm}\label{thm:index LP}
Let $1 \in P \subset A$ be an inclusion of unital $C^*$-algebras with a finite Watatani index 
and $E \colon A \rightarrow P$ be a faithful conditional expectation.
Suppose that $A$ has the LP property and $E$ has the Rokhlin property.
Then $P$ has the LP property.
\end{thrm}

\begin{proof}
Let $x \in P$ and $\varepsilon > 0$. 
Since $A$ has the LP property, $x$ can be approximated by a line sum of projection $\sum_{i=1}^n\lambda_ip_i$
\ $(p_i \in A)$ such that $\|x - \sum_{i=1}^n\lambda_ip_i\| < \varepsilon$.

Since $\beta\colon A^\infty \rightarrow P^\infty$ is an injective *-homomorphism by Proposition~\ref{prp:embedding}, we have 
$$
\|\beta(x - \sum_{i=1}^n\lambda_ip)\| = \|\beta(x) - \sum_{i=1}^n\lambda_i\beta(p_i)\| < \varepsilon.
$$
Since ${\beta_{|P}} = id$, we have $\|x - \sum_{i=1}^n\lambda_i\beta(p_i)\| < \varepsilon$.
Each projection in $P^\infty$ can be lifted to a projection in $\ell^\infty(\NN, P)$, so 
we can find a set of projections $\{q_i\}_{i=1}^n \subset P$ such that
$$
\|x - \sum_{i=1}^n\lambda_iq_i\| < \varepsilon.
$$
Therefore, $P$ has the LP property.
\end{proof}

\vskip 3mm

\begin{thrm}\label{thm:crossed product LP}
Let $\alpha$ be an action of a finite group $G$ on a simple unital $C^*$-algebra $A$ and 
$E$ be the canonical conditional expectation from $A$ onto the fixed point algebra $P = A^{\alpha}$ defined by 
$$
E(x) = \frac{1}{\#G}\sum_{g\in G}\alpha_g(x) \quad \text{for} \ x \in A, 
$$
where $\#G$ is the order of $G$.
Suppose that $\alpha$ has the Rokhlin property. 
We have, then, if $A$ has the LP property, the fixed point algebra and the crossed product $A \rtimes_\alpha G$ 
have the LP property.
\end{thrm}

\vskip 3mm 

Before giving the proof we need the following two {lemmas}, which must be  well-known.

\begin{lemma}\label{lem:basic construction}
Under the same conditions in Theorem~\ref{thm:crossed product LP} 
consider the following two basic constructions:

\begin{align*}
A^\alpha &\subset A \subset C^*\langle A, e_P\rangle \subset C^*\langle B, e_A\rangle \ (B = C^*\langle A, e_P\rangle)\\
{(A^\alpha)} &\subset A \subset A \rtimes_\alpha G \subset 
C^*\langle A \rtimes_\alpha G, e_F\rangle, 
\end{align*}
where $F \colon A \rtimes_\alpha G \rightarrow A$ is a canonical conditional expectation.
Then there is an isomorphism $\pi \colon C^*\langle A, e_P\rangle \rightarrow A \rtimes_\alpha G$ 
and $\tilde{\pi} \colon C^*\langle B, e_A\rangle \rightarrow 
C^*\langle A \rtimes_\alpha G, e_F\rangle$
such that
\begin{enumerate}
\item $\pi(a) = a \ \forall a \in A$,
\item $\pi(e_p) = q$, where $q = \dfrac{1}{|G|}\sum_{g\in G}u_g$,
\item $A \rtimes_\alpha G = C^*\langle A, q\rangle$,
\item $\tilde{\pi}(b) = \pi(b) \ \forall b \in B$,
\item $\tilde{\pi}(e_A) = e_F$.
\end{enumerate}
Moreover, we have 

(6) $F \circ \pi = \hat{E}$ and $\pi \circ \hat{\hat{E}} = \hat{F} \circ \tilde{\pi}$.
\end{lemma}

\vskip 3mm

\begin{proof}
At first we prove the condition $(3)$.
Since $\alpha$ is outer, $\alpha$ is saturated by \cite[Proposition~4.9]{JP}, 
that is, 
\begin{align*}
A \rtimes_\alpha G& = \overline{\mathrm{linear\  span \ of}\ 
\{\sum_{g\in G}(\alpha_g(x)u_g)^*(\sum_{g\in G}\alpha_g(y)u_g)\mid x, y \in A\}}\\
&= \overline{\mathrm{linear\  span \ of}\ \{\dfrac{1}{|G|}\sum_{g\in G}x^*\alpha_g(y)u_g\mid x, y \in A\}}.
\end{align*}

On the contrary, for any $x, y \in A$
\begin{align*}
xqy &= x\dfrac{1}{|G|}\sum_{g\in G} u_gy\\
&= \dfrac{1}{|G|}x\sum_{g\in G}u_gyu_g^*u_g\\
&= \dfrac{1}{|G|}\sum_{g\in G}x\alpha_g(y)u_g,
\end{align*}
hence $A \rtimes_\alpha G = {C^*\langle A, q\rangle}$.

Since for any $a \in A$
\begin{align*}
qaq &= \dfrac{1}{|G|^2}\sum_{g\in G}u_g a\sum_{h\in G}u_h\\
&= \dfrac{1}{|G|^2}\sum_{g, h \in G}u_gau_g^*u_gh\\
&= \dfrac{1}{|G|^2}\sum_{g, h\in G}\alpha_g(a)u_{gh}\\
&= \dfrac{1}{|G|}\sum_{g\in G}\alpha_g\dfrac{1}{|G|}\sum_{h\in G}u_{gh}\\
&= E(x)q,
\end{align*}
by Remark~\ref{rmk:b-const}(5)
there is an isomorphism $\pi \colon C^*\langle A, e_P\rangle \rightarrow C^*\langle A, q\rangle = A\rtimes_\alpha G$
such that $\pi(a) = a$ for any $a \in A$ and $\pi(e_P) = q$. Hence the conditions $(1)$ and $(2)$ are proved.

By the similar steps we will show the conditions $(4)$ and $(5)$. 
Since for any $x, y, a, b \in A$ 
\begin{align*}
(e_F\pi(xe_Py)e_F)(aqb) &= (e_Fxqy)F(\dfrac{1}{|G|}\sum_{g\in G}a\alpha_g(b)u_g)\\
&= \dfrac{1}{|G|}(e_Fxqyab)\\
&= \dfrac{1}{|G|^2}xyab
\end{align*}
On the contrary, 
\begin{align*}
\pi(\hat{E}(xe_Py))e_F(aqb) &= \pi(\dfrac{1}{|G|}xy)\dfrac{1}{|G|}ab\\
&= \dfrac{1}{|G|^2}xyab. \\
\end{align*}
Hence, we have $e_F\pi(xe_py)e_F = \pi(\hat{E}(xe_Py))$. 
By Remark~\ref{rmk:b-const}(5) there is an isomorphism 
$\tilde{\pi}\colon C^*\langle B, e_A\rangle \rightarrow 
C^*\langle A \rtimes_\alpha G, e_F\rangle$ 
such that $\tilde{\pi}(b) = \pi(b)$ for any $b \in B$ and $\tilde{\pi}(e_A) = e_F$.

The condition $(6)$ comes from the direct computation.
\end{proof}

\vskip 3mm

\begin{lemma}\label{lemma:matrix}
Under the same conditions in Lemma~\ref{lem:basic construction}
$C^*\langle A \rtimes_\alpha G, e_F\rangle$ is isomorphic to $M_{|G|}(A)$.
\end{lemma}

\vskip 3mm

\begin{proof}
Note that $\{(u_g^*, u_g)\}_{g\in G}$ is a quasi-basis for $F$. 
By \cite[Lemma~3.3.4]{Watatani:index} there is an isomorphism from 
$C^*\langle A \rtimes_\alpha G, e_F\rangle$
to $rM_{|G|}(A)r$, where $r = [E(u_g^*u_h)]_{g, h \in G} = I_{|G|}$. 
Hence $C^*\langle A \rtimes_\alpha G, e_F\rangle$ is isomorphic to $M_{|G|}(A)$.
\end{proof}

\vskip 3mm

Proof of Theorem~\ref{thm:crossed product LP}:
%\begin{proof}

Let $\{e_g\}_{g\in G}$ be the Rokhlin projection of $E$.
From Proposition~\ref{prp:group} $E\colon A \rightarrow A^G$ is of index finite and has a projection $e \in A' \cap A^\infty$ 
such that $E^\infty(e) = \dfrac{1}{|G|}1$. Note that $\mathrm{Index} E = |G|$ and $e = e_1$.
Consider the basic construction 
$$
A^G \subset A \subset C^*\langle A, e_P\rangle \subset C^*\langle B, e_A\rangle \ \ (B = C^*\langle A, e_P\rangle)
$$

Since $A$ is simple, the map 
$A \ni x \mapsto xe$ is injective, hence we know that $E$ has the Rokhlin property. 
Therefore $A^G$ has the LP property by Theorem~\ref{thm:index LP}.

Since $C^*\langle A \rtimes_\alpha G, e_F\rangle$ is isomorphic to $M_{|G|}(A)$ by Lemma~\ref{lemma:matrix} 
and $A$ has the LP property,
$C^*\langle A, e_F\rangle$ has the LP property.
Hence, $C^*\langle B, e_A\rangle$ has the LP property, 
because that  $C^*\langle A \rtimes_\alpha G, e_F\rangle$ is isomorphic to $C^*\langle B, e_A\rangle$ 
from Lemma~\ref{lem:basic construction}. 
From Proposition~\ref{prp:Rokhlin}
$\hat{\hat{E}}\colon C^*\langle B, e_A\rangle \rightarrow C^*\langle A, e_P\rangle$ 
has the Rokhlin property, hence we conclude that $C^*\langle A, e_P\rangle$ has the LP property by 
Theorem~\ref{thm:index LP}. Since $C^*\langle A, e_P\rangle$ is isomorphhic to $A \rtimes_\alpha G$ by 
Lemma~\ref{lem:basic construction}, we conclude that $A \rtimes_\alpha G$ has the LP property.

\hfill$\qed$
%\end{proof}

\vskip 3mm

\begin{rem}\label{NotLP}
\begin{enumerate}
\item
When an action $\alpha$ of a finite group $G$ does not have the Rokhlin property, 
we have an example of simple unital $C^*$-algebra with the LP property such that 
the fixed point algebra $A^G$ does not have the LP property by Example~\ref{exa:fixed point}.
Note that the action $\alpha$ does not have the Rokhlin property.
\item
When an action of a finite group $G$ on a unital $C^*$-algebra $A$ has the Rokhlin property, the crossed product 
can be locally approximated by the class of matrix algebras over corners of $A$ $($\cite[Theorem~3.2]{OP}$)$. 
Many kind of properties are preserved by this method such that AF algebras \cite{NP}, AI algebras, 
AT algebras, simple AH algebras with slow dimension growth and real rank zero \cite{OP}, 
D-absorbing separable unital $C^*$-algebras for a strongly self-absorbing $C^*$-algebras \cite{HW}, 
simple unital separable strongly self-absorbing $C^*$-algebras \cite{OT}, unital Kirchberg $C^*$-algebras \cite{OP} etc.
Like the ideal property \cite{PP}, however, since the LP property is not preserved by passing to corners by 
Lemma~\ref{lem:LP}, we can not apply this method to determine the LP property of the crossed products.
\end{enumerate}
\end{rem}
 
 \vskip 3mm

We could also  have many examples which shows that 
the LP property is preserved under the formulation of crossed products 
from the following observation.

\vskip 3mm

Let $A$ be an infinite dimensional simple $C^*$-algebra and let $\alpha$ be an action 
from a finite group $G$ on $\mathrm{Aut}(A)$. 
Recall that $\alpha$  has the tracial Rokhlin property if for every finite set $F \subset A$, 
every $\varepsilon > 0$,
and every positive element $x \in A$ with $\|x\| = 1$, there are mutually orthogonal
projections $e_g \in A$ for $g \in G$ such that:
\begin{enumerate}
\item[$(1)$] 
$\|\alpha_g(e_h) - e_{gh}\| < \varepsilon$ for all $g, h \in G$ and all $a \in F$,
\item[$(2)$]
$\|e_ga - ae_g\| < \varepsilon$ for all $g \in G$ and all $a \in F$.
\item[$(3)$] With $e = \sum_{g\in G}e_g$, 
the projection $1 - e$ is Murray-von Neumann equivalent
to a projection in the hereditary subalgebra of $A$ generated by $x$.
\item[$(4)$] With $e$ as in (3), we have $\|exe\| > 1 - \varepsilon$.
\end{enumerate}

It is obvious that the tracial Rokhlin property is weaker than the Rokhlin prpoperty.

\vskip 3mm

\begin{prop}
Let $\alpha$ be an action of a finite group $G$ on a simple unital $C^*$-algebra $A$ with a unique tracial state. 
Suppose that $\alpha$ has the tracial Rokhlin property. 
If $A$ has the LP property, then the crossed product $A\rtimes_\alpha G$ has the LP property.
\end{prop}

\begin{proof}
From \cite[Proposition~5.7]{ELPW} the restriction map from tracial states on the crossed product $A \rtimes_\alpha G$ 
to $\alpha$-invariant tracial states on $A$ is isomorphism. Hence, $A \rtimes_\alpha G$ has a unique tracial state.

Since $\alpha$ is a pointwise outer (i.e., for any $g \in G\backslash\{0\}$ $\alpha_g$ is outer) by \cite[Lemma~1.5]{NP},
$A \rtimes_\alpha G$ is simple.

Therefore, by \cite[Corollary~4]{GP} $A \rtimes_\alpha G$ has the LP property.
\end{proof}

\vskip 3mm

\begin{rem}
There are many examples of actions $\alpha$ of finite groups on simple unital $C^*$-algebras with real rank zero  and 
a unique tracial state such that  $\alpha$ has the tracial Rokhlin property. See \cite{NP}, \cite{ELPW}.
\end{rem}

%\end{example}
%---------------------------
%{\it Acknowledgements.} %The authors are partially supported by NAFOSTED, Vietnam. 
%The second author would like to thank professor George A. Elliott for many discussions.

\bibliographystyle{amsplain}

\end{document}